\documentclass{amsart}
\usepackage{amssymb,amsmath,amsthm}
\usepackage{graphicx}
\usepackage{multicol}
\usepackage{array}
\usepackage{tikz}
\usepackage{tkz-graph}
\usepackage{pgf}

\usetikzlibrary{arrows,shapes}

\newcommand{\iso}{\cong}
\newcommand{\Z}{\mathbb{Z}}

\newcommand{\C}{\mathbb{C}}

\newcommand{\N}{\mathbb{N}}

\newcommand{\nor}{\trianglelefteq}
\newcommand{\abs}[1]{\left\vert #1 \right\vert}

\newcommand{\op}[1]{#1^{\text{op}}}
\newcommand{\Endo}[1]{\text{End}\left(#1\right)}
\DeclareMathOperator{\Adutomorphism}{Aut}
\newcommand{\Aut}[1]{\Adutomorphism\left(#1\right)}

\DeclareMathOperator{\Adjoint}{Adj}
\newcommand{\Adj}[1]{\Adjoint\left( #1 \right)} 
\DeclareMathOperator{\Isom}{Isom}
\newcommand{\PIsom}[1]{\Psi\Isom\left( #1 \right)} 

\newtheorem{thm}[equation]{Theorem}
\newtheorem{cor}[equation]{Corollary}
\newtheorem{lem}[equation]{Lemma}

\newtheorem{proposition}[equation]{Proposition}
\newtheorem{definition}[equation]{Defintion}

\theoremstyle{definition}

\theoremstyle{remark}
\newtheorem{remark}[equation]{Remark}

\numberwithin{equation}{section}

\begin{document}
\title{Longer Nilpotent Series for Classical Unipotent Subgroups}
\date{March 10, 2015}
\author{Joshua Maglione}
\address{Department Of Mathematics, Colorado State University, Fort Collins, CO 80523}
\email{maglione@math.colostate.edu}

\begin{abstract}
In studying nilpotent groups, the lower central series and other variations can be used to construct an associated $\Z^+$-graded Lie ring, which is a powerful method to inspect a group. Indeed, the process can be generalized substantially by introducing $\N^d$-graded Lie rings.
We compute  the adjoint refinements of the lower central series of the unipotent subgroups of the classical Chevalley groups over the field $\Z/p\Z$ of rank $d$. We prove that, for all the classical types, this characteristic filter is a series of length $\Theta(d^2)$ with nearly all factors having $p$-bounded order.
\end{abstract}

\maketitle

\section{Introduction}
The connection between $p$-groups and Lie rings has long been known and continues to be a symbiotic relationship. Indeed, in \cite{Lazard54}, Lazard proves that, for a series of a group $G$,
\[ G = G_1 \geq G_2 \geq \cdots \geq G_n\geq G_{n+1} = 1,\]
if $[G_i,G_j]\leq G_{i+j}$ for all $i,j\geq 1$, then there is an associated graded Lie ring to the series given by 
\[ L = \bigoplus_{i=1}^n G_i/G_{i+1}.\]

J. B. Wilson weakened the hypothesis of Lazard's statement by replacing series with filters, and proved filters still have an associated graded Lie ring \cite{Wilson13b}. A filter $\phi$ is a function from a pre-ordered commutative monoid (see Definition \ref{pre-order definition}) $(M,\prec)$ into the set of normal subgroups of $G$ satisfying 
\[ (\forall m,n\in M) \qquad\qquad  [\phi_m,\phi_n] \leq \phi_{m+n} \quad \text{and} \quad m\prec n \text{ implies } \phi_m\geq \phi_n.\]
Filters produce lattices of normal subgroups, and in the case where $(M,\prec)$ is totally ordered, the filter is a series. 

A notable feature of filters is their ease of refinement. Given a filter from $M$ into the normal subgroups of $G$, one can insert new subgroups into the lattice and generate a new filter. Wilson gives a few locations to search for new subgroups to add to filters, one of which is the adjoint refinement, which uses ring theoretic properties stemming from the graded Lie ring product.

The length of the adjoint filter, the filter stabilized by the adjoint refinement, seems difficult to predict without considering specific examples. Therefore, we work with the well known family of unipotent subgroups of classical groups, e.g. the group of upper unitriangular matrices. Surprisingly, we find something new. We compute the adjoint refinements of the lower central series of the unipotent subgroups of these groups over the field $\Z/p\Z$. After the refinement process stabilizes, we find that the factors of this new filter are very small. Moreover, the subgroups of this filter are totally ordered, so we obtain a characteristic series. The length of the adjoint series is largely unchanged under most quotients, so this seemingly narrow case of examples is a great place to start understanding these filters on a wide range of families of $p$-groups.

\begin{thm}\label{main result for A}
If $U=\langle x_r(t): r\in\Phi^+, t\in\Z/p\Z\rangle$ is a subgroup of the Chevalley group $A_d(\Z/p\Z)$ for $p \geq 3$ (i.e. the group of upper unitriangular matrices), then all of the (nontrivial) factors of the series stabilized by the adjoint refinement process have order $p$ or $p^2$.
\end{thm}

For comparison, the usual lower central series (which is equal to the lower central exponent $p$ series) has large factors, many of order approximately $p^{d}$.
When investigating the action of $\Aut{U}$ on $U$, Theorem \ref{main result for A} puts a large constraint on the possible actions of the automorphism group. Indeed, in the computations for isomorphism or automorphism testing, a reduction in the order of the first factor alone can greatly reduce the algorithm run time \cite{LEGO2002}. For example, the automorphism group of the first factor of the lower central series of $U$ is $\text{GL}(d,p)$. However, for the series of Theorem \ref{main result for A}, the automorphism group of the first factor is $\text{GL}(e,p)$, where $e$ is either $1$ or $2$. For unipotent subgroups of other classical Chevalley groups, we get a similar outcome. 

\begin{thm}\label{main result}
Let $p$ be a prime with $p\geq 3$. Let $U=\langle x_r(t): r\in\Phi^+, t\in\Z/p\Z\rangle$ be a unipotent subgroup of a classical Chevalley group over $\Z/p\Z$ with Lie rank $d$. There exists a characteristic series of $U$ whose length is $\Theta(d^2)$ and whose factors have constant order (except possibly a constant number of factors). Furthermore, the associated Lie algebra, $L(\alpha)$ is $\N^m$-graded, where $m$ is either $\left\lceil d/2 \right\rceil$ or $\left\lfloor d/2 \right\rfloor$. 
\end{thm}

\begin{cor}\label{cor:Aut(U)-series}
Let $U$ and $m$ be as in Theorem \ref{main result}. Then $U/U'$ has an $\Aut{U}$ invariant series of length $m$. There is at most one factor with dimension $1$, and if $U$ is of type $D$, then there is a factor of dimension $3$. All other factors have dimension $2$.
\end{cor}

As usual with these groups, the case $p=2$ requires additional care. Modest changes need to be made for characteristic two and are addressed in Remark \ref{p=2}, at the end of Section \ref{the alpha series section}. In addition, we have exclusively used the field $\Z/p\Z$, but we expect similar results for arbitrary finite fields. In this case, the order of the factors depends on the size of the field. For comments on how to generalize this approach to arbitrary finite fields, see Remark \ref{generalize to q} at the end of Section \ref{the alpha series section}.

The paper is organized as follows. We discuss the definition of a filter and necessary information in Section \ref{section of filters}. In Section \ref{general construction}, we give a method to compute an adjoint refinement. In Section \ref{the alpha series section}, we prove our main results, Theorem \ref{main result for A} and Theorem \ref{main result}.

\subsection{Notation}
We let $\N$ and $\Z^+$ denote the set of nonnegative integers and positive integers respectively. For a set $S$, we denote the power set of $S$ by $2^S$. For a group $G$ and for $x,y\in G$, we let $[x,y]$ denote $x^{-1}y^{-1}xy$. In general, we define $[x_1]=x_1$ and $[x_1,x_2,...,x_n,x_{n+1}]=[[x_1,...,x_n],x_{n+1}]$. If $H,K\leq G$, then $[H,K]=\left\langle [h,k]:h\in H,k\in K\right\rangle$. We use the same recursive notation for subgroups of $G$ as we do for elements of $G$. Throughout the paper, $p$ is a prime, and $\Z_p$ denotes the field $\Z/p\Z$.

We adopt the same notation for root systems and Chevalley groups as provided by Carter in \cite[Chapters 2 -- 4]{Carter72}. That is, we let $\Phi$ denote a system of roots. Define an ordering of the roots and let $\Phi^+$ and $\Phi^-$ denote the positive and negative roots respectively. Let $\Pi$ be the set of fundamental roots of $\Phi$, and let $\{ h_r: r\in\Pi\}\cup\{ e_s: s\in\Phi\}$ be a Chevalley basis for the Lie algebra $\mathfrak{g}$ over $\C$ for some Cartan decomposition. 

The Chevalley group of type $\mathfrak{g}$ over $\Z_p$, denoted $\mathfrak{g}(p)$, is the group of automorphisms of the Lie algebra $\mathfrak{g}_{\Z_p}=\mathfrak{g}\otimes_\Z\Z_p=\langle x_r(t) : r\in\Phi,t\in\Z_p\rangle$, where $x_r(t)=\text{exp}(t\text{ ad }e_r)$.
The root subgroup of $r\in\Phi$ is $X_r=\langle x_r(t) : t\in\Z_p\rangle$. The maximal unipotent subgroups of $\mathfrak{g}(p)$ are all conjugate to the group $U=\langle X_r : r\in \Phi^+\rangle$, so in our proofs we take $U=\langle X_r : r\in \Phi^+\rangle$. Note that unipotent subgroups stay the same in all the classical groups, including special and projective, so we do not need to be specific. 

\section{Filters}\label{section of filters}
Since Theorem \ref{main result for A} and Theorem \ref{main result} use a filter generation algorithm, we summarize the necessary ideas from \cite[Section 3]{Wilson13b} for the sake of completeness.
\begin{definition}\label{pre-order definition}
A \emph{pre-order} $\prec$ on a commutative monoid, $M$, is a reflexive and transitive relation, and for all $k,\ell,m,n \in M$, if $k\prec \ell$ and $m\prec n$, then $k+m\prec \ell +n$.
\end{definition}

\begin{definition}
A \emph{filter} of $G$ is a function $\phi: M\rightarrow 2^G$ such that for all $m,n\in M$, $\phi_m=\phi(m)$ is a subgroup of $G$, 
\[  [ \phi_m,\phi_n] \leq \phi_{m+n} \qquad \text{and}\qquad m\prec n \,\text{ implies }\, \phi_m\geq \phi_n.\]
\end{definition}

We remark that filters $\phi:\N\rightarrow 2^G$, with $\phi_0=G$, are exactly the \emph{$N$-series} introduced by Lazard \cite{Lazard54}. From the definition of a filter, $\phi_m\nor G$ for all $m\in M$.

Every filter induces a new filter $\partial \phi : M \rightarrow 2^G$ given by
\[ \partial\phi_m = (\partial\phi)_m = \prod_{s\in M-\{0\}} \phi_{m+s} = \langle \phi_{m+s} : s\in M-\{0\}\rangle .\]
It follows that for each $m,n\in M$,
\[ [ \partial\phi_m, \phi_n ] = \prod_{s\in M-\{0\}} [ \phi_{m+s},\phi_n ] =\prod_{s\in M-\{0\}} \phi_{m+s+n}\leq \prod_{s\in M-\{0\}} \phi_n\cap\phi_{m+s}\leq  \phi_{m+n}. \]
In particular, $[\partial\phi_m,\phi_m]\leq \phi_m$; thus, $\partial\phi_m\nor\phi_m$.
For $m\in M$, let
\begin{equation}\label{Ls}
L_m=\phi_m/\partial\phi_m.
\end{equation} 
Thus, by \cite[Theorem 3.3]{Wilson13b}, the abelian group, 
\begin{equation}\label{LieRing}
L(\phi)=\bigoplus_{m\in M} L_m,
\end{equation} is a Lie ring with product on the homogeneous components 
\begin{equation}\label{Lie ring product}
(\forall x\in\phi_s,\forall y\in\phi_t)\qquad \qquad [\partial\phi_sx,\partial\phi_ty]=\partial\phi_{s+t}[x,y].
\end{equation}
Note that if $\phi$ is a filter such that $\phi$ produces an $N$-series of $G$, then $\partial\phi$ is also an $N$-series and $\partial\phi_m=N_{m+1}$. In that case, $L(\phi)$ is the Lie ring described by Lazard cf. \cite[Theorem 2.1]{Lazard54}. 

Suppose $\mathcal{S}$ generates $M$ as a monoid, and $0\in\mathcal{S}$. Let $\mathcal{G}=\mathcal{G}(M,\mathcal{S})$ be the (directed) Cayley graph whose vertices are $M$ and whose labeled edge set is $\{ m\overset{s}\longrightarrow n: m+s=n, s\in\mathcal{S}\}$. Furthermore, let $\mathcal{G}_m^n$ denote the set of all paths, ${\bf t}$, from vertex $m$ to vertex $n$ in $\mathcal{G}$. We write a path ${\bf t}$ as a sequence of edge labels the path traverses. That is, for each $s_i\in\mathcal{S}$,  ${\bf t} =(s_1,...,s_k)$ where $m+ s_1+\cdots +s_k = n$. Suppose $\pi : \mathcal{S} \rightarrow 2^G$ is a function, and for simplicity, we denote $[\pi_{s_1},...,\pi_{s_k}]$ by $[\pi_{\bf t}]$ if ${\bf t} = (s_1,...,s_k)$.
Define a new function $\bar{\pi}:M\rightarrow 2^G$ by
\begin{equation}\label{pi bar} 
\bar{\pi}_m=\prod_{{\bf t}\in\mathcal{G}_0^m} [ \pi_{\bf t} ].
\end{equation}

We close this section with sufficient conditions on $M$, $\mathcal{S}$ and $\pi$ so that $\bar{\pi}$ is a filter.
\begin{definition}
A \emph{generalized refinement monoid} $(M,\prec)$ is a commutative pre-ordered monoid with minimal element $0$ and if $m\prec n$ and $n=\sum_{i=1}^rn_i$, then there exists $m_i\prec n_i$ where $m=\sum_{i=1}^rm_i$.
\end{definition}

\begin{thm}[{Wilson \cite[Theorem 3.9]{Wilson13b}}]\label{generating filter}
Suppose $(M,\prec)$ is a generalized refinement monoid and $0\in\mathcal{S}\subset M$ is an interval closed generating set of $M$. If $\pi:\mathcal{S}\rightarrow 2^G$ is an order reversing function into the set of normal subgroups of $G$, then $\bar{\pi}$ is a filter.
\end{thm}

\section{Constructing the Stable Adjoint Series}\label{general construction}

For Theorems \ref{main result for A} and \ref{main result}, we will use the adjoint series introduced by Wilson \cite[Section 4]{Wilson13b}. To produce the adjoint series (or \emph{$\alpha$-series}), we iterate a refinement process until it stabilizes. There are two major steps in computing the adjoint series: finding the new subgroups to add at the top of the series and generating the lower terms of the series. The former is nearly independent from the latter, so we can find all the new subgroups that get added to the top section (successively) before we start to generate the lower terms. This is essentially how we approach our investigation of the adjoint series in Section \ref{the alpha series section}. Figure \ref{process} gives a visualization of the basic process. 

\begin{figure}
\begin{center}
\includegraphics[scale=0.6]{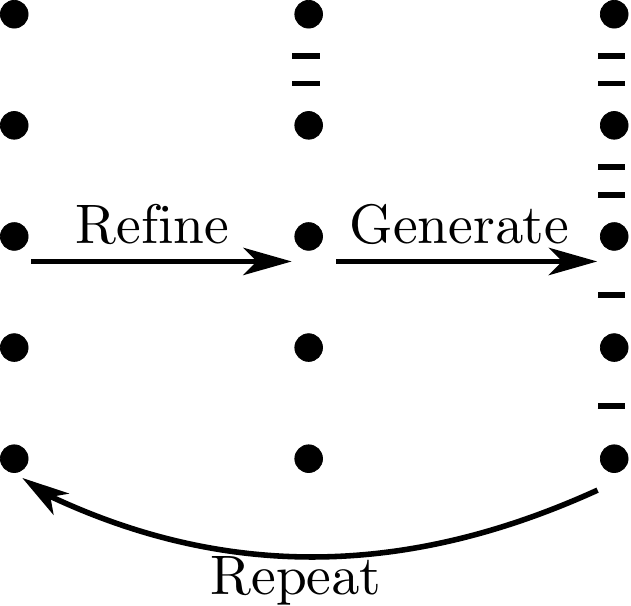}
\caption{Iterating the adjoint refinement process.}\label{process}
\end{center}
\end{figure}

We begin by describing how we obtain new subgroups. Let $\phi:(\N,\leq)\rightarrow 2^G$ be a filter, and set $\alpha^{(1)}_n=\phi_n$ for all $n\in \N$. In our construction, we use the lower central series of $G$ as our initial filter, which is indexed by $\Z^+$. This allows for the opportunity to record operators (e.g. the holomorph of $G$) at the top of the filter, $\phi_0$, even though we presently take $\phi_0$ to be $G$. We define the $\alpha^{(1)}$-series to be the filter $\alpha^{(1)}:\N\rightarrow 2^G$, and in general, the $\alpha^{(k)}$-series is a filter $\alpha^{(k)}:(\N^k,\prec) \rightarrow 2^G$ inductively produced as follows, where $\prec$ is the lexicographic order. 

As established in Section \ref{section of filters}, if $n\in\N^k$, then $L_n^{(k)} = \alpha^{(k)}_n / \partial\alpha^{(k)}_{n}$
is a homogenous component of the associated $\N^k$-graded Lie algebra $L(\alpha^{(k)})$ cf. (\ref{Ls}) and (\ref{LieRing}). To obtain the $\alpha^{(k+1)}$-series from the $\alpha^{(k)}$-series for $k\geq 1$, we use the graded product map $\circ:L_{s}^{(k)}\times L_{t}^{(k)}\rightarrow L_{s+t}^{(k)}$ given by communtation in $G$ cf. (\ref{Lie ring product}).
\begin{definition}
The \emph{adjoint ring of} $\circ$ is
\begin{align*}
\Adj{\circ}=&\left\{ (f,g)\in\Endo{L_{s}^{(k)}}\times\op{\Endo{L_{t}^{(k)}}}: \right.\\
&\qquad  \left. \forall u\in L_{s}^{(k)}, \forall v\in L_{t}^{(k)}, uf\circ v=u\circ gv \right\}.
\end{align*}
\end{definition}
This is our source of new (characteristic) subgroups. Another characterization of the adjoint ring of $\circ$ is to define it as the ring for which $\circ$ factors through $\otimes_{\Adj{\circ}}:L_s^{(k)}\times L_t^{(k)}\rightarrow L_s^{(k)}\otimes_{\Adj{\circ}}L_t^{(k)}$ uniquely. This latter characterization implies that the properties of $\Adj{\circ}$ influence the properties of $\circ$, and hence commutation in $G$. See \cite[Section 2]{Wilson13a} for further details on adjoints. 

We choose $(s,t)\in\N^k\times\N^k$ to be the lex least pair where the Jacobson radical of $\Adj{\circ}$ is nontrivial. Let $J$ be the Jacobson radical of $\Adj{\circ}$, and let $J^0=\Adj{\circ}$. Recursively define $J^{i+1}=J^iJ$ for all $i\in\N$. For all $i\in\N$, define $H_i$ so that $\alpha_{s}^{(k)}\geq H_i\geq \partial\alpha_{s}^{(k)}$ and
\begin{equation} \label{alphas}
H_i/\partial\alpha_{s}^{(k)} = L_{s}^{(k)}J^i.
\end{equation}
If $J=0$, then we get no new subgroups. If $J=0$ for each $(s,t)\in\N^k\times\N^k$, then the $\alpha^{(k)}$-series has no nontrivial adjoint refinement. 

To incorporate these new subgroups into a filter, we first obtain a generating set for $\N^{k+1}$ which includes the indices of the new subgroups. Let 
\begin{equation}\label{Genset} 
\mathcal{S}_{k+1}=\{ (n,i)\in\N^k\times\N : n\preceq s\},
\end{equation} 
so that $\mathcal{S}_{k+1}$ is interval closed and generates $\N^{k+1}$. For $(n,i)\in\N^k\times\N$, define
\begin{equation}\label{pi for a2}  
\pi^i_n=\left\{ \begin{array}{ll} \alpha_{n}^{(k)} & \text{if } n\prec s, \\ H_{i} & \text{if } n = s.\end{array}\right.
\end{equation}
Observe that $\pi$ is a function from $\mathcal{S}_{k+1}$ into the normal subgroups of $G$ (which is totally ordered with respect to $\prec$ the lexicographic order) which satisfies the conditions of Theorem \ref{generating filter}. Thus, $\bar{\pi}:\N^{k+1}\rightarrow 2^G$ is a filter, and we set $\alpha^{(k+1)}=\bar{\pi}$. We refer to the filter $\alpha^{(k+1)}$ as the $\alpha^{(k+1)}$-series.

We show that the adjoint series is a characteristic series.
\begin{proposition}\label{Characteristic}
If the initial filter, $\phi:\N\rightarrow 2^G$, is a characteristic series, then the adjoint series of $G$ is a characteristic series.
\end{proposition}

To prove the proposition, we show that $\Aut{G}$ acts on $\Adj{\circ}$ via conjugation. From \cite[Proposition 3.8]{Wilson09}, $\Aut{G}$ maps into the \emph{pseudo-isometries} of $\circ$, which are defined to be 
\[ \PIsom{\circ}=\left\{ (h,\hat{h})\in\Aut{L_{s}^{(k)}}\times\Aut{L_{2s}^{(k)}}: xh\circ hy = (x\circ y)\hat{h}\right\}.\] 
The pseudo-isometries of $\circ$ act on $\Adj{\circ}$ by conjugation as the following lemma proves.

\begin{lem}\label{PIsom(M)}
$\PIsom{\circ}$ acts on $\Adj{\circ}$ by
\[ (f,g)^{(h,\hat{h})}=(h^{-1}fh,hgh^{-1})\in\Adj{\circ},\]
for $(f,g)\in\Adj{\circ}$ and $(h,\hat{h})\in\PIsom{\circ}$. Furthermore, this action is faithful.
\end{lem}
\begin{proof} 
Let $x,y\in L_{e_1}^{(k)}$. It follows that $\PIsom{\circ}$ acts on $\Adj{\circ}$ as 
\[ xh^{-1}fh \circ y = (xh^{-1}f\circ h^{-1}y)\hat{h} = (xh^{-1}\circ gh^{-1}y)\hat{h} = x\circ hgh^{-1}y. \]
It follows that this action is faithful because $h$ and $\hat{h}$ are automorphisms cf. \cite[Proposition 4.16]{Wilson09}.
\end{proof}

\begin{proof}[Proof of Proposition \ref{Characteristic}] 
By \cite[Proposition 3.8]{Wilson09}, $\Aut{G}$ maps into $\PIsom{\circ}$. Therefore, by Lemma \ref{PIsom(M)}, $\Aut{G}$ acts on $\Adj{\circ}$ by conjugation. Furthermore, since $J$ is the intersection of all maximal ideals in $\Adj{\circ}$, it follows that the action of $\Aut{G}$ fixes $J$. Thus, $J^i$ is fixed by the action of $\Aut{G}$ for every $i\in\Z^+$. Therefore, $L_{e_1}^{(k)}J^n$ is characteristic, and hence, $\pi_n$ is characteristic for $n\in\N^{k+1}$, provided $\alpha^{(k)}_{(n_1,...,n_k)}$ is characteristic. Since $\phi_m$ is characteristic for $m\in\N$, it follows by induction that each term in the $\alpha^{(k+1)}$-series is characteristic.
\end{proof}

\section{The Stable Adjoint Refinement of Classical Unipotent Groups}\label{the alpha series section}
In this section we prove Theorems \ref{main result for A} and \ref{main result}. Most of our work goes into proving Theorem \ref{main result for A}; we will see that Theorem \ref{main result} follows from the proof of Theorem \ref{main result for A}. Recall that the adjoint series is a refinement of some other (characteristic) series. Let $U\leq \mathfrak{g}(p)$ be a unipotent subgroup. Our initial series is the lower central series and we denote the $k^{th}$ term of the series by $\gamma_k(U)$ or $\gamma_k$. We let $\gamma_0(U)=U$, so that $\N$ indexes our filter. Let $\circ$ be the graded product map given in (\ref{Lie ring product}).

For details on Chevalley groups, root systems, and Lie algebras see \cite{Carter72}. Recall the Chevalley commutator formula
\begin{thm}[Chevalley]
Let $u,t\in\Z_p$ and $s,r\in\Phi$. Then for each $i,j>0$ with $ir+js\in\Phi$, there exists constants $C_{ijrs}$ such that 
\begin{equation}\label{commutation}
[x_s(u),x_r(t)] = \prod_{i,j>0}x_{ir+js}(C_{ijrs} (-t)^iu^j)
\end{equation}
where the product is taken in increasing order of $i+j$.
\end{thm} 
The details for the constants $C_{ijrs}$ can be found in \cite[p. 77]{Carter72}. 
We denote the fundamental roots, $p_i$, to be consistent with the Dynkin diagram for $\mathfrak{g}$. 
That is, $p_1$ is connected to $p_2$, $p_2$ is connected to both $p_1$ and $p_3$, and so on. 
The Dynkin diagrams for different types of root systems are given in Figure \ref{Dd}.
\begin{figure}
\begin{center}\begin{tabular}{m{1cm}m{6.2cm}}
$A_d$: &
\begin{tikzpicture}
\GraphInit[vstyle=Hasse]
\SetGraphUnit{1.5}
\tikzset{VertexStyle/.style = {draw,
                               shape          = \VertexShape,
                               color          = \VertexLineColor,
                               fill           = \VertexLightFillColor,
                               inner sep      = 0pt,
                               text           = \VertexTextColor,
                               minimum size   = 5pt,
                               line width     = 1pt}}
\Vertex[x=0,y=0]{p1}
\Vertex[x=1,y=0]{p2}
\Vertex[x=2,y=0]{p3}
\Vertex[x=4,y=0]{pd2}
\Vertex[x=5,y=0]{pd1}
\Vertex[x=6,y=0]{pd}
\Edge[](p1)(p2)
\Edge[](p2)(p3)
\Edge[](pd2)(pd1)
\Edge[label=$\cdots$](p3)(pd2)
\Edge[](pd1)(pd)
\end{tikzpicture}\\
$B_d$: &
\begin{tikzpicture}
\GraphInit[vstyle=Hasse]
\SetGraphUnit{1.5}
\tikzset{VertexStyle/.style = {draw,
                               shape          = \VertexShape,
                               color          = \VertexLineColor,
                               fill           = \VertexLightFillColor,
                               inner sep      = 0pt,
                               text           = \VertexTextColor,
                               minimum size   = 5pt,
                               line width     = 1pt}}
\Vertex[x=0,y=0]{p1}
\Vertex[x=1,y=0]{p2}
\Vertex[x=2,y=0]{p3}
\Vertex[x=4,y=0]{pd2}
\Vertex[x=5,y=0]{pd1}
\Vertex[x=6,y=0]{pd}
\Edge[](p1)(p2)
\Edge[](p2)(p3)
\Edge[](pd2)(pd1)
\Edge[label=$\cdots$](p3)(pd2)
\path[thick] (pd1.north) edge (pd.north);
\path[thick] (pd1.south) edge (pd.south);
\end{tikzpicture}\\
$C_d$: &
\begin{tikzpicture}
\GraphInit[vstyle=Hasse]
\SetGraphUnit{1.5}
\tikzset{VertexStyle/.style = {draw,
                               shape          = \VertexShape,
                               color          = \VertexLineColor,
                               fill           = \VertexLightFillColor,
                               inner sep      = 0pt,
                               text           = \VertexTextColor,
                               minimum size   = 5pt,
                               line width     = 1pt}}
\Vertex[x=0,y=0]{pd}
\Vertex[x=1,y=0]{pd1}
\Vertex[x=2,y=0]{pd2}
\Vertex[x=4,y=0]{p3}
\Vertex[x=5,y=0]{p2}
\Vertex[x=6,y=0]{p1}
\Edge[](p1)(p2)
\Edge[](p2)(p3)
\Edge[](pd2)(pd1)
\Edge[label=$\cdots$](p3)(pd2)
\path[thick] (pd1.north) edge (pd.north);
\path[thick] (pd1.south) edge (pd.south);
\end{tikzpicture}\\
$D_d$: &
\begin{tikzpicture}
\GraphInit[vstyle=Hasse]
\SetGraphUnit{1.5}
\tikzset{VertexStyle/.style = {draw,
                               shape          = \VertexShape,
                               color          = \VertexLineColor,
                               fill           = \VertexLightFillColor,
                               inner sep      = 0pt,
                               text           = \VertexTextColor,
                               minimum size   = 5pt,
                               line width     = 1pt}}
\Vertex[x=0,y=0]{p1}
\Vertex[x=1,y=0]{p2}
\Vertex[x=2,y=0]{p3}
\Vertex[x=4,y=0]{pd2}
\Vertex[x=5,y=0]{pd1}
\Vertex[x=6,y=0.5]{pdtop}
\Vertex[x=6,y=-0.5]{pdbot}
\Edge[](p1)(p2)
\Edge[](p2)(p3)
\Edge[](pd2)(pd1)
\Edge[label=$\cdots$](p3)(pd2)
\Edge[](pd1)(pdtop)
\Edge[](pd1)(pdbot)
\end{tikzpicture}
\end{tabular}\end{center}\caption{The Dynkin diagrams for the different classical types of root systems.}\label{Dd}
\end{figure}
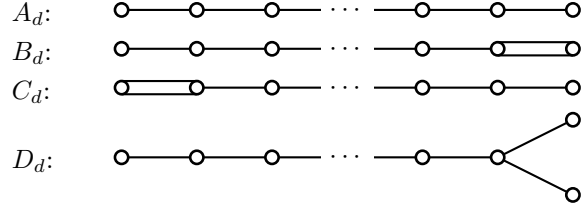
We choose an ordering of the fundamental roots $\Pi$ so that $p_i\prec_{\Pi} p_j$ if $i<j$. Thus, if $(r,s)$ is an extra-special pair of roots (i.e. $r+s\in\Phi$, $0\prec_{\Pi} r\prec_{\Pi} s$ and for all $r_1+s_1=r+s$, $r\preceq_{\Pi} r_1$), then $[e_r,e_s]=-(v+1)e_{r+s}$; cf. \cite[p. 58]{Carter72}.

Note that for every $r\in\Phi^+$, we can write 
\begin{equation}\label{decomp}
r=p_{i_1}+\cdots +p_{i_k},
\end{equation} 
for (not necessarily distinct) $p_{i_j}\in\Pi$. Thus, we may talk about the \emph{height} of each (positive) root $r$ denoted $h(r)$ which is the sum of the integer coefficients of $r$ when written as in (\ref{decomp}). We say a root subgroup $X_r$ has height $m$ if $h(r)=m$. Let $U_m$ denote the subgroup generated by all $X_r$ such that $h(r)\geq m$. As the next lemma states, these subgroups almost always coincide with the lower central series of $U$. 

\begin{lem}[{Spitznagel \cite[Theorem 1]{Spitznagel68}}]\label{Gibbs Lemma}
Let $U$ be a maximal unipotent subgroup of a classical Chevalley group over $\Z_p$. If $U$ is of type $B$ or $C$ and $p=2$, then $\gamma_2(U) \ne U_2 = \langle X_r : h(f) \geq 2, r\in \Phi^+\rangle$. Otherwise, $\gamma_m(U)=U_m$ for all $m$. 
\end{lem}

Because of Lemma \ref{Gibbs Lemma}, we assume $p\geq 3$ for types $B$ and $C$. Some comments about when $p=2$ are given at the end of the section, see Remark \ref{p=2}. From Lemma \ref{Gibbs Lemma}, it follows that $L^{(1)}_i=U_i/U_{i+1}$ is the quotient containing all root subgroups of height $i$. Since there are exactly $d$ fundamental roots, $L_1^{(1)}\iso \Z_p^d$. Furthermore, since there are exactly $d-1$ positive roots with height 2, $L_2^{(1)}\iso\Z_p^{d-1}$. Therefore, $\circ$ may be regarded as an alternating $\Z_p$-bilinear map, $\circ: \Z_p^d\times\Z_p^d\rightarrow\Z_p^{d-1}$.

We construct the structure constants $M_{\mathfrak{g}}$ (Gram matrix) of $\circ$ for the classical types $\mathfrak{g}$, i.e. for all $u,v\in \Z_p^d$, $u\circ v =uM_{\mathfrak{g}}v^t$. Observe that for type $A$,
\begin{equation}\label{structure constants}
[x_{p_i}(s),x_{p_j}(t)]=\left\{ \begin{array}{ll} x_{p_i+p_j}(st) & \text{if } (p_i,p_j)\text{ is extra special},\\ x_{p_i+p_j}(-st) & \text{if }(p_j,p_i)\text{ is extra special},\\ 0 & \text{otherwise}.\end{array}\right.
\end{equation}
Define $\varphi_1: L_1\rightarrow \Z_p^d$ by $\varphi_1(\gamma_2x_{p_i}(t))=te_i$, and define $\varphi_2:L_2\rightarrow \Z_p^{d-1}$ by $\varphi_2(\gamma_3x_{p_i+p_j}(t))=te_{j-1}$, provided $i<j$. Note that both $\varphi_1$ and $\varphi_2$ are vector space isomorphisms. 
Thus, by (\ref{structure constants}), the structure constants matrix for $A_d(p)$ is 
\begin{equation}\label{matrix M}
M_A=\begin{bmatrix}
0 & e_1 &  & &\\[5pt]
-e_1 & 0 & e_2& &\\
 & -e_2 & 0 &\ddots &\\
 & &\ddots&\ddots &e_{d-1}\\[5pt]
 & & &-e_{d-1} &0
\end{bmatrix}.\end{equation}

\begin{lem}\label{initial results}
Let $M_\mathfrak{g}$ be the structure constants for $\circ$ of $\mathfrak{g}(p)$. Then $M_\mathfrak{g}$ has the same shape as the Cartan matrix of the root system of type $\mathfrak{g}$.
\end{lem}
\begin{proof} 
This follows from the Chevalley commutator formula (\ref{commutation}).
\end{proof}

Hence, the structure constant matrices $M_B$ and $M_C$ are equal to $M_A$, provided they are the same rank. Finally, for type $D$, we have 
\begin{equation}\label{TypeD}
 M_D = \begin{bmatrix}
0 & e_1 &  & & &\\
-e_1 & 0 & \ddots& & &\\
 &\ddots &\ddots&e_{d-3} & &\\
& &-e_{d-3} &0 &e_{d-2} &e_{d-1}\\
& & & -e_{d-2}& 0& 0\\
& & & -e_{d-1}& 0& 0
\end{bmatrix}.
\end{equation}

Now we compute the adjoint rings of these bilinear maps. By Lemma \ref{initial results}, the adjoint rings for $M_A$, $M_B$, and $M_C$ are the same, but the adjoint ring for $M_D$ is different. For each $i\in\{1,...,d-1\}$, let $M_i\in\text{M}_{d}(\Z_p)$ to be the matrix with $1$ in the $(i,i+1)$ entry, $-1$ in the $(i+1,i)$ entry, and 0 elsewhere. Similarly, let $N\in \text{M}_d(\Z_p)$ be the matrix with $1$ in the $(d-2,d)$ entry and $-1$ in the $(d,d-2)$ entry, and $0$ elsewhere. Therefore, $M_A=\sum_{i=1}^{d-1}e_iM_i$ and $M_D=\sum_{i=1}^{d-2}e_iM_i + e_{d-1}N$. It follows that
\begin{align*}
\Adj{M_1}&=\left\{\left(\begin{bmatrix}w & x & *\\ y& z & * \\ 0&0&*\end{bmatrix},\begin{bmatrix}z & -x & *\\ -y& w & * \\ 0&0&*\end{bmatrix}\right): w,x,y,z\in\Z_p\right\} \text{ and}\\
\Adj{N} & =\left\{\left(\begin{bmatrix}* & 0 & * & 0\\ * & w & * & x \\ * & 0 &* &0 \\ * & y & * & z \end{bmatrix},\begin{bmatrix}* & 0 & * & 0\\ * & z & * & -x \\ * & 0 &* &0 \\ * & -y & * & w \end{bmatrix}\right): w,x,y,z\in\Z_p\right\}.
\end{align*}

Note that we can obtain $M_i$ from $M_1$ by applying a permutation. Applying such a permutation also permutes the adjoint ring, and therefore
\begin{equation*}
\Adj{M_i}=\left\{\left(\begin{bmatrix}* & 0 & 0 & *\\ *& w & x & * \\ * &y & z &*\\ * & 0 & 0 & * \end{bmatrix},\begin{bmatrix}* & 0 & 0 & *\\ *& z & -x & * \\ * &-y & w &*\\ * & 0 & 0 & * \end{bmatrix}\right): w,x,y,z\in\Z_p\right\}.
\end{equation*}
Observe that $\Adj{M_A}=\Adj{\sum_{i=1}^{d-1}e_iM_i}=\bigcap_{i=1}^{d-1}\Adj{M_i}$. For $x,y,z\in\Z_p$, let $D(x,y)\in\text{M}_d(\Z_p)$ denote the diagonal matrix with diagonal $(x,y,x,y,...)$, $Q_A(x,y)=xE_{12}+yE_{(d-1)d}$, and $Q_D(x,y,z)=xE_{12}+yE_{(d-2)(d-1)}+zE_{(d-2)d}$. Thus,
\begin{align*} 
\Adj{M_A} &= \{ ( D(w,x) + Q_A(y,z), D(x,w) - Q_A(y,z)) :w,x,y,z\in\Z_p\} \\
\Adj{M_D} &= \{ ( D(v,w) + Q_D(x,y,z), D(w,v) - Q_D(x,y,z) ) : v,w,x,y,z\in\Z_p\}.
\end{align*}

Now we describe the Jacobson radical of the adjoint rings of $M_A$ and $M_D$; denote the radicals $J_A$ and $J_D$ respectively. If $d\geq 3$ (or $d\geq 4$ for type $D$ as $D_3$ and $A_3$ are the same root systems), then
\begin{align} \label{Jacobson types ABC}
J_A &=\left\{ (Q_A(y,z),-Q_A(y,z)) :y,z\in\Z_p \right\}\\ \label{Jacobson type D}
J_D &=\left\{ (Q_D(x,y,z),-Q_D(x,y,z)) : x,y,z\in\Z_p\right\}.
\end{align}
Otherwise (if $d\leq 2$), the adjoint ring of the forms
\[ [0] \quad\text{and}\quad \begin{bmatrix} 0 & e_m \\ -e_m & 0 \end{bmatrix}\]
is $\Z_p$ and $\text{M}_2(\Z_p)$ (resp.), which both have trivial Jacobson radical. Note that $J_A^2=J_D^2=0$. Thus we get at most one new subgroup at the top of the series corresponding to $L_1J_A$ (or $L_1J_D$). In fact, the new subgroup is
\begin{equation}\label{H_1} 
H_1= \left\{\begin{array}{ll} \langle X_{p_1},X_{p_d},\gamma_2\rangle & \text{if type } A,B, \text{ or } C, \\ \langle X_{p_1},X_{p_{d-1}}, X_{p_d},\gamma_2\rangle & \text{if type } D,\end{array}\right.
\end{equation}
and $H_i=\gamma_2$ for all $i\geq 2$ c.f. (\ref{alphas}).

For each iteration of the computation, we use the bilinear map, given by commutation, from the first nonzero factor of the $\alpha^{(k)}$-series. This bilinear map depends on $U/H_1$ and $\gamma_2/\partial\gamma_2$. We remark that the quotient $\gamma_2/\partial\gamma_2$ in the second iteration contains all root subgroups $X_r$ with $h(r)=2$ where both $r-p_1,r-p_d\not\in\Phi^+$ (additionally $r-p_{d-1}\not\in\Phi^+$ if type $D$). Note that if $\widetilde{U}$ is the unipotent subgroup of $A_{d-2}(p)$ ($A_{d-3}(p)$ if type $D$), then commutation in $U/H_1$ is the same as commutation in $\widetilde{U}/\gamma_2(U)$ (up to relabeling). Thus, the structure constants for $*:U/H_1\times U/H_1\rightarrow \gamma_2(U)/\partial\gamma_2(U)$, is given by (\ref{matrix M}) for all types, except the structure constants have smaller dimension:
\begin{equation*} M_* = \begin{bmatrix}
0 & e_2 &  & \\
-e_2 & 0 & \ddots& \\
 &\ddots &\ddots&e_{d-2} \\
& &-e_{d-2} &0 
\end{bmatrix} \quad\text{or}\quad 
\begin{bmatrix}
0 & e_2 &  & \\
-e_2 & 0 & \ddots& \\
 &\ddots &\ddots&e_{d-3} \\
& &-e_{d-3} &0 
\end{bmatrix}.
\end{equation*}

If $U$ is the unipotent subgroup of the classical group $\mathfrak{g}(p)$, then after the first iteration the structure constants matrix of commutation is given by (\ref{matrix M}), except with smaller dimension. Thus, for the subsequent iterations, we cut the dimension by two and the structure constants matrix is similar to (\ref{matrix M}). This proves the following proposition since each iteration adds a dimension to the monoid which indexes the associated Lie algebra.

\begin{proposition}\label{grading}
The associated Lie algebra of the adjoint series is $\N^m$-graded, with $m=\left\lceil\frac{d}{2}\right\rceil$ (types $A$, $B$, or $C$) or $m=\left\lfloor\frac{d}{2}\right\rfloor$ (type $D$).
\end{proposition}

Because the bilinear maps essentially stay the same, we can easily list the top of the adjoint series. Let $H_1$ be defined as in (\ref{H_1}) and $H_k=\langle X_{p_k},X_{p_{d-k+1}},H_{k-1}\rangle$ for $k\geq 2$ (if type $D$, then $H_k=\langle X_{p_k},X_{p_{d-k}}, H_{k-1}\rangle$). With $m$ given by Proposition \ref{grading},
\begin{equation}\label{series} 
U \geq H_{m-1} \geq \cdots \geq H_1 \geq \gamma_2 \geq \cdots \geq \gamma_c \geq 1,
\end{equation}
without filter generation. Observe that if $U$ is of type $A$, $B$, or $C$, then
\[ \left\vert U/H_{m-1} \right\vert = \left\{\begin{array}{ll} p & \text{if } 2\nmid d,\\ p^2 & \text{if } 2\mid d,\end{array}\right. \]
and $\left\vert H_{k+1}/H_k\right\vert = \left\vert H_1/\gamma_2\right\vert =p^2$, for $1\leq k \leq m-2$. On the other hand, if $U$ is of type $D$, then 
\[ \left\vert U/H_{m-1} \right\vert = \left\{\begin{array}{ll} p^2 & \text{if } 2\nmid d,\\ p & \text{if } 2\mid d,\end{array}\right. \]
$\left\vert H_{k+1}/H_k\right\vert = p^2$, for $1\leq k \leq m-2$, and $\left\vert H_1/\gamma_2\right\vert = p^3$. Therefore, Corollary \ref{cor:Aut(U)-series} follows.

Now we investigate the lower terms of the adjoint series of these unipotent groups. Currently, the series in (\ref{series}) has $\Theta(d)$ terms and it may not be a filter. To get $\Theta(d^2)$ terms, we must generate a filter with the series in (\ref{series}). Let $\mathcal{S}_m = \{ n\in\N^m : n\prec 2e_1\}$. We recursively define a function $\pi_m:\mathcal{S}_m\rightarrow 2^G$. Let $\pi_2:\mathcal{S}_2\rightarrow 2^G$ where 
\[ \pi_2(n,i) =\left\{ \begin{array}{ll} \gamma_1 & \text{if } (n,i)\preceq (1,0), \\ H_1 & \text{if } (n,i)=(1,1), \\ \gamma_2 & \text{otherwise.}\end{array}\right.\]
Thus, for $k\geq 3$, let $\pi_k:\mathcal{S}_k\rightarrow 2^G$ where 
\[ \pi_k(n,i) =\left\{ \begin{array}{ll} \pi_{k-1}(n) & \text{if } (n,i)\preceq (e_1,0), \\ H_{k-1} & \text{if } (n,i)=(e_1,1), \\ H_{k-2} & \text{otherwise.}\end{array}\right.\]
To be consistent with the established notation, let $\pi_m = \pi$.

Because of the lexicographic ordering, it is possible to have an infinite number of indices correspond to the same image under $\alpha$. That is, for $n\in\N^m$, the cardinality of $\{ n'\in\N^m : \alpha_{n'}=\alpha_n\}$ need not be finite. It is even possible for the previous set to include elements with different $n_1$ values. Thus when referring to a term in the $\alpha$-series, we use the smallest (lex) index $(n_1,...,n_m)$, where $n_1$ is as large as possible. In the case of the last term in the series, we let $n_1$ equal one plus the class of $U$. 

To get the lower terms, we must generate them via (\ref{pi bar}), so for all $n\in\N^m$, 
\[ \alpha_n=\alpha_n^{(m)}=\bar{\pi}_n=\prod_{{\bf t}\in\mathcal{G}_0^n}[\pi_{\bf t}]. \]
Note that the terms of the commutator subgroups are equal to either $U$, $H_{i}$, or $\gamma_2$. 
The following lemma states that we don't need to run through all ${\bf t}\in\mathcal{G}_0^n$.

\begin{lem}\label{min paths}
Let $k\geq 2$ and $n=(n_1,...,n_k)$ where $n_i\in\N$. Then
\[ \alpha^{(k)}_n = \prod_{\bf t} [\pi_{\bf t}] \gamma_{n_1+1},\]
where the product runs through all paths, ${\bf t}$, of length $n_1$ from $\mathcal{G}^n_0$.
\end{lem}
\begin{proof} 
Recall that $\mathcal{S}_k=\{n\in\N^k : n\prec 2e_1\}$. Therefore, 
\[ \alpha^{(k)}_n \geq [\pi_{e_1},...,\pi_{e_1},\pi_{(0,n_2,...,n_k)}] =\gamma_{n_1+1}=\gamma_{n_1+1},\]
so every path of length at least $n_1+1$ from $\mathcal{G}^n_0$ is already contained in $\alpha^{(k)}_n$. Since $(n_1,0,...,0)\preceq n$, it follows that $\gamma_{n_1}\geq \alpha_n^{(k)}$, so the statement follows.
\end{proof}

From Lemma \ref{min paths}, it follows that, for a fixed $n\in\N^m$, the entries and their multiplicities (the number of occurrences) of the commutator $[\pi_{\bf t}]$ are completely determined. Indeed, for $n=(n_1,...,n_m)$ and for all commutators $n_1$ entries, the entry $H_i$ must have multiplicity $n_{i+1}$. Because $\N^k$ is a commutative monoid, we get every possible order of terms in the commutator. 

In the following theorem, we examine the multiset of entries from $[\pi_{\bf t}]$. These multisets are partially ordered under component inclusion. That is, for multisets $\mathcal{A}$ and $\mathcal{B}$, where $\abs{\mathcal{A}}=\abs{\mathcal{B}}=e$, if there exist sequences $\{A_i\}^e_{i= 1}$ and $\{B_i\}^e_{i= 1}$ with $A_i\leq B_i$ for each $i$ and $\mathcal{A}=\bigcup_{i=1}^eA_i$ and $\mathcal{B}=\bigcup_{i=1}^eB_i$, then $\mathcal{A}\leq \mathcal{B}$. 

Let $U\leq A_d(p)$ be a maximal unipotent subgroup. For fixed $r\in\Phi^+$, let $\alpha_n$ be the smallest subgroup of the adjoint series of $U$ such that $X_r\leq \alpha_n$. Let $\mathcal{M}_r$ be the collection of the multisets $\mathcal{A}$ of entries from $[\pi_{\bf t}]$, where $X_r\leq[\pi_{\bf t}]$ and $\abs{\mathcal{A}}=n_1$. Then $\mathcal{M}_r$ has a minimal element. Indeed, if $r=p_i+\cdots + p_j$, using symmetry, we may assume $i\leq m$, then the minimal element is 

\begin{equation*}
\mathcal{B}=\left\{ \begin{array}{ll} \{H_i,...,H_{m-1},U,U,H_{m-1},...,H_{d-j+1}\} & \text{if $d$ is even and } i\leq m < j,\\ \{H_i,...,H_{m-1},U,H_{m-1},...,H_{d-j+1}\} & \text{if $d$ is odd and } i\leq m < j,\\ \{H_i,...,H_{j}\} & \text{if } i<j<m. \end{array}\right.
\end{equation*}

If $\alpha_n$ is the smallest term of the adjoint series containing $X_r$ and $\mathcal{B}$ is the smallest multiset of $\mathcal{M}_r$, then 
\[ \alpha_n = \prod_{\{b_1,...,b_{n_1}\}=\mathcal{B}} [b_1,...,b_{n_1}].\]
This observation is critical to the proof of Theorem \ref{main result for A}.

\subsection{Proof of Theorem \ref{main result for A}}
\begin{proof}[Proof of Theorem \ref{main result for A}]
Let $r\in\Phi^+$ and write $r=p_i+\cdots +p_j$. Using the notation above, the smallest multiset $\mathcal{B}$ of $\mathcal{M}_r$ is given by the equation above. Observe that if $r'=p_{d-j+1}+\cdots +p_{d-i+1}$ then $\mathcal{B}$ is the smallest multiset in $\mathcal{M}_{r'}$. Therefore, if $\alpha_n$ is the smallest term of the adjoint series containing $X_r$, then $\alpha_n$ is the smallest term in the adjoint series containing $X_{r'}$. 

We show that if $s\in\Phi^+$, where $s\ne r$ and $s\ne r'$, and if $\alpha_{n'}$ is the smallest term in the adjoint series containing $X_s$, then $\alpha_n\ne \alpha_{n'}$. If $\mathcal{B}_r$ and $\mathcal{B}_s$ are the smallest multisets contained in $\mathcal{M}_r$ and $\mathcal{M}_s$ respectively, then by the above equation for $\mathcal{B}$, we have that $\mathcal{B}_r\ne\mathcal{B}_s$.
 Therefore, $\alpha_n\ne \alpha_{n'}$. Hence, the orders of the factors of the adjoint series of $U$ are either $p$ (if $r=r'$) or $p^2$.
\end{proof}

\begin{cor}
The adjoint series of $U\leq A_d(p)$ has $\Theta(d^2)$ factors. 
\end{cor}
\begin{proof} 
Each factor has order $p$ or $p^2$ and $\log_p|U|=\binom{d+1}{2}$.
\end{proof}

\begin{remark}
There are similar statements like that of Theorem \ref{main result for A} for the other classical types. In fact, for types $B$ and $C$, the factor orders are either $p$ or $p^2$, with an exception in the middle of the series, where one term has order roughly $p^{d/2}$. Furthermore, for type $D$, the factor orders are either $p$, $p^2$, or $p^3$; again with an exception in the middle of the series where one term has order roughly $p^{d/2}$. 
\end{remark}

\subsection{Proof of Theorem \ref{main result}}

\begin{proof}[Proof of Theorem \ref{main result}] The proof of Theorem \ref{main result for A} applies to all roots $r$ with $h(r)\leq d$ and whose summands are unique. Every root system $\Phi$ with $\abs{\Pi}=d$, contains roots of the form $r=p_i+\cdots +p_{i+h-1}$ where $1\leq i$ and $i+h-1\leq d$. Thus, the length of the adjoint series for $\mathfrak{g}(p)$ is at least as long as the length of the adjoint series for $A_d(p)$.
\end{proof}

\begin{remark}\label{p=2}
When $p=2$, some modest changes can be made to the previous lemmas and theorems, aside from the change from alternating to symmetric bilinear maps. While the structure constants of the bilinear map, $\circ : U/U'\times U/U'\rightarrow U'/\gamma_3(U)$, are different in characteristic two for types $B$ and $C$ \cite[p. 848]{Gazdanova06}, the Jacobson radical of $\Adj{\circ}$ is similar to that of type $D$. Thus, we obtain factors of order $p^3$ in this case. 
\end{remark}

\begin{remark}\label{generalize to q}
To generalize these results to field extensions of $\Z_p$, we use centroids to find the appropriate field extension without requiring it as input. As introduced by Wilson in \cite{Wilson12}, the \emph{centroid} of a bilinear map $\circ : V \times V \rightarrow W$ is defined to be 
\[ \text{Cent}(\circ) = \{ (f,h)\in \Endo{V}\times\Endo{W} : \forall u,v\in V, uf\circ v = u\circ fv = (u\circ v)h \}.\]
If $U\leq \mathfrak{g}(K)$, then structure constants of $M_\mathfrak{g}$ are similar to that of Lemma \ref{initial results}. Instead of $1 \times 1$ blocks along the upper and lower diagonal, we get $e\times e$ blocks where the field has size $p^e$, and a calculation shows that Cent$(M_\mathfrak{g})\iso K$. In this case, the sizes of nearly every factor of the adjoint series is $|K|$-bounded.
\end{remark}

The statement of the main theorem is not explicit about the length of the adjoint series. In particular, it is unknown if the adjoint series is significantly longer than the lower central series for small rank. However, we see in Figure \ref{data} that the length of the lower central series is much smaller than the adjoint series, even for small ranks. Included in that figure are the lengths of the adjoint series for the exceptional types ($F$ and $E$) for $p\geq 3$. The adjoint ring of the bilinear map for $G_2(p)$ has a trivial Jacobson radical for all $p\geq 5$, and hence, has no nontrivial adjoint refinement. These computations were run in {\sc Magma} \cite{Magma}. 

\begin{figure}
\centering
\includegraphics[width=0.9\textwidth]{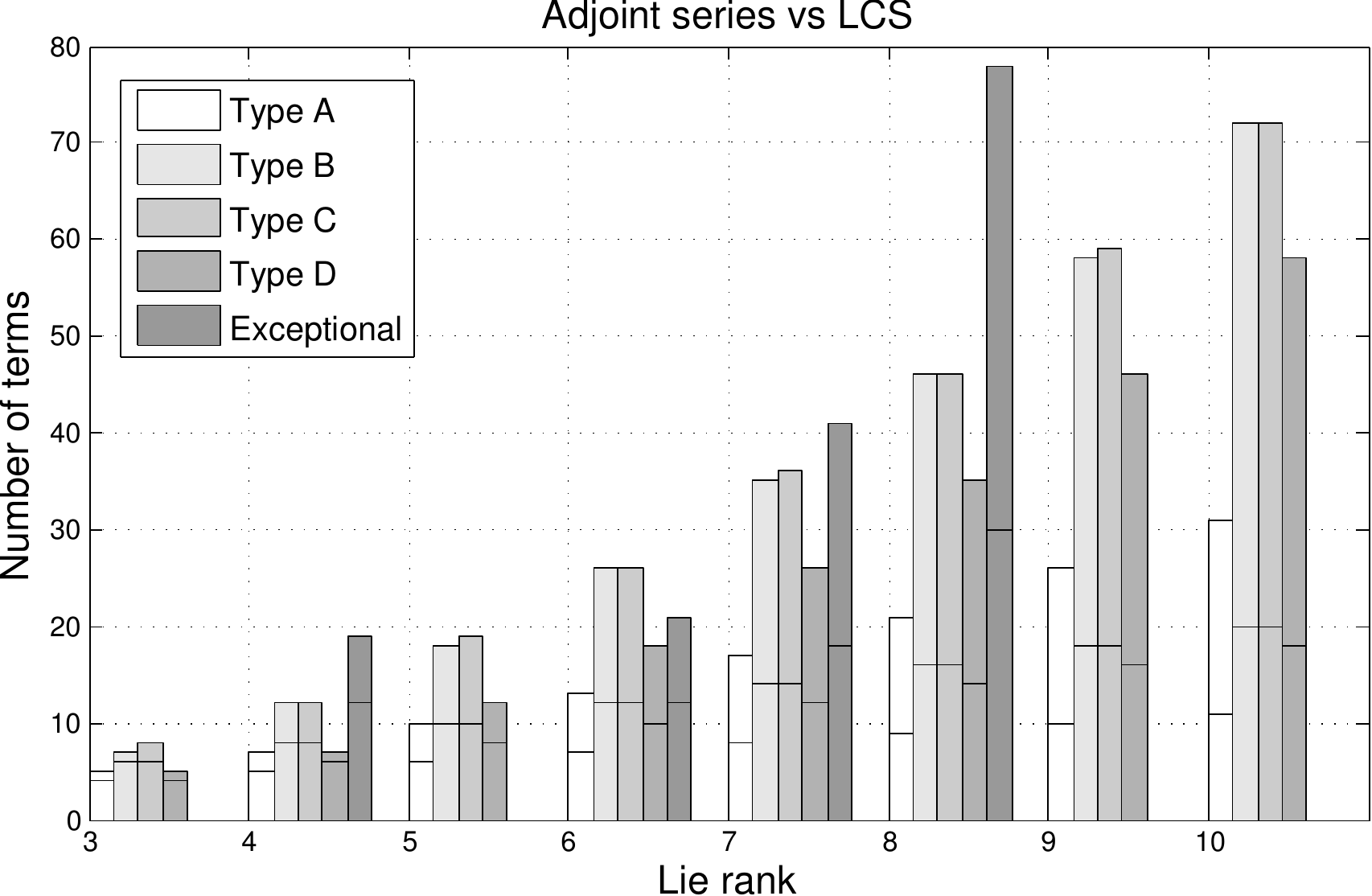}
\caption{A comparison between the lower central series (lower bar) and the adjoint series (upper bar) for small Lie rank.}\label{data}
\end{figure}

\section*{Acknowledgments}
The author is indebted to J. B. Wilson for suggesting and advising this project, A. Hulpke for helpful discussions and coding advice, and the referee for insightful comments which improved the clarity of the paper.

\bibliographystyle{plain}
\bibliography{LongerSeries}

\end{document}